\documentclass[reqno]{amsart}
\usepackage{diagrams}
\usepackage{amssymb}
\usepackage{mathrsfs}
\usepackage{cite}
\usepackage{tikz-cd}
\usepackage{MnSymbol}
\usepackage{a4wide}
 
\usepackage{xcolor}

\DeclareMathOperator\C{\mathbb C}

\DeclareMathOperator\Z{\mathbb Z}
\newcommand{\Disk}{\mathbb D}
\newcommand{\Om}{\Omega}

\newtheorem{theorem}{Theorem}[section]
\newtheorem{lemma}[theorem]{Lemma}
\newtheorem{cor}[theorem]{Corollary}
\newtheorem{prop}[theorem]{Proposition}
\theoremstyle{definition}

\theoremstyle{remark}
\newtheorem{remark}[theorem]{Remark}

\newcommand{\dontprint}[1]\relax

\newcommand{\bos}{\operatorname{bos}}

\newcommand{\La}{\Lambda}

\renewcommand{\P}{{\mathbb P}}
\newcommand{\A}{{\mathbb A}}
\newcommand{\wt}{\widetilde}
\newcommand{\ot}{\otimes}
\newcommand{\SL}{\operatorname{SL}}

\renewcommand{\AA}{{\mathcal A}}

\newcommand{\TT}{{\mathcal T}}
\newcommand{\DD}{{\mathcal D}}
\newcommand{\CC}{{\mathcal C}}

\newcommand{\HH}{{\mathcal H}}
\newcommand{\OO}{{\mathcal O}}

\renewcommand{\SS}{{\mathcal S}}

\newcommand{\de}{\delta}
\newcommand{\sub}{\subset}

\newcommand{\ov}{\overline}

\newcommand{\om}{\omega}
\newcommand{\la}{\lambda}
\renewcommand{\a}{\alpha}
\renewcommand{\b}{\beta}

\newcommand{\id}{\operatorname{id}}

\renewcommand{\th}{\theta}
\newcommand{\coker}{\operatorname{coker}}

\newcommand{\hra}{\hookrightarrow}

\newcommand{\De}{\Delta}

\newcommand{\pa}{\partial}

\numberwithin{equation}{section}
\title{On supercurves of genus $1$ with an underlying odd spin structure}
\author{Alexander Polishchuk}
\thanks{Partially supported by the NSF grant DMS-2001224, 
and within the framework of the HSE University Basic Research Program}

\address{
    Department of Mathematics, 
    University of Oregon, 
    Eugene, OR 97403, USA; and National Research University Higher School of Economics, Moscow, Russia
  }
  \email{apolish@uoregon.edu}

\begin{document}

\begin{abstract} We study the standard family of supercurves of genus $1$ with an underlying odd spin structures. We give a simple algebraic description of
this family and of the compactified family of stable supercurves with one Neveu-Schwarz puncture.
We also describe the Gauss-Manin connection on the $1$st de Rham cohomology of this family,
and compute the superperiods of global differentials.
\end{abstract}

\maketitle

\section{Introduction}

The moduli spaces of supercurves (also known as SUSY curves) is an important geometric ingredient
for the superstring perturbation theory (see e.g., \cite{Witten}). 
Recall that a supercurve over a superscheme $S$ is a smooth
proper morphism $\pi:X\to S$ of relative dimension $1|1$, together with a relative distribution $\DD\sub \TT_{X/S}$ of rank $0|1$, which is everywhere
non-integrable, i.e., the supercommutator induces an isomorphism $\DD\ot \DD\rTo{\sim} \TT_{X/S}/\DD$. We refer to \cite{DW}, \cite{FKP}, \cite{LB-R} 
for background on supercurves.
The moduli space of supercurves, possibly with some Neveu-Schwarz punctures (i.e., sections $S\to X$), 
has two connected components corresponding to whether the underlying spin structure is {\it even} or {\it odd}.

We refer to \cite{FR} and \cite{Levin1} for an overview of the genus $1$ case.
In this paper we consider the case of supercurves of genus $1$ with an underlying odd spin structure (this means that the underlying spin bundle on an elliptic curve 
is trivial). For every $\tau$ in the upper half-plane, there is a corresponding supercurve $X_\tau\to \C^{0|1}$ of genus $1$ with the underlying elliptic curve $\C/(\Z+\Z\tau)$
equipped with the trivial spin structure. In fact, there is a standard family $X\to S$ inducing all these supercurves, where $S$ is the product of the upper half-plane with $\C^{0|1}$,
obtained as the quotient of the relative $\C^{1|1}$ (see Sec.\ \ref{stand-sec}). This family is equipped with a natural Neveu-Schwarz (NS) puncture.
Complementing known results on such supercurves (see \cite{Levin}, \cite{FR}, \cite{Rabin}), we determine
the algebra $\OO(X_\tau-p)$, where $p\sub X_\tau$ is a natural puncture. We show that this leads to a simple algebraic description of $X_\tau$ as a blow-up
of a cubic in the superprojective plane $\P^{2|1}$ deforming the classical Weierstrass model (see Theorem \ref{algebra-thm}).
Furthermore, we use this model to extend this family to a family $\ov{X}\to \ov{S}$ of stable supercurves with one NS puncture. 
We compute the Kodaira-Spencer map for the extended family
proving that the corresponding map $\ov{S}\to \ov{\SS}^{odd}_{1,1}$ is an isomorphism on tangent spaces (see Theorem \ref{stable-thm}).

The second topic of this paper is a calculation of the Gauss-Manin connection.
It is a classical fact that the cohomology group $H^1_{dR}(C)$ of a smooth projective curve over $\C$ can be computed as 
the space of meromorphic differentials of the second kind modulo the exact ones (see \cite{AH}) and the Gauss-Manin connection can be calculated
in these terms (see \cite{Manin-conn}). In particular, for a point $p\in C$,
this cohomology group can be described as the quotient $\om_C(C-p)/d(\OO(C-p))$. 

We present a similar isomorphism for smooth proper supercurves
and describe explicitly the Gauss-Manin connection in these terms.
More precisely, given a section $p\sub X$ of $\pi:X\to S$, we construct an isomorphism (assuming $S$ is affine) of the relative de Rham cohomology $\HH^1_{X/S}$ with
$\om_{X/S}(X-p)/\de(\OO(X-p))$,
where $\om_{X/S}$ is the relative Berezinian line bundle 
and 
$$\de:\OO_X\to \om_{X/S}$$
is the canonical derivation associated with $\DD\sub \TT_{X/S}$ (see Proposition \ref{H1-prop}).
We then give a recipe for computing the Gauss-Manin connection on $\HH^1_{X/S}$ in these terms (see Proposition \ref{GM-recipe-prop}).
Using this recipe, we compute explicitly the Gauss-Manin connection on the de Rham cohomology space $\HH^1_{X/S}$ 
for the standard family $\pi:X\to S$ of genus $1$ supercurves with an odd underlying spin-structure (see Theorem \ref{GM-thm}). We also compute the superperiods of the global differentials for this family, i.e., the map $\pi_*\om_{X/S}\to \HH^1_{X/S}$ (see Corollary \ref{per-cor}).

\bigskip

\noindent
{\it Acknowledgment}.
This paper is an offshoot of a joint ongoing project with Giovanni Felder and David Kazhdan and moduli of supercurves (its results will be used in the study of superperiods
near one of the boundary components of the moduli of supercurves of genus $2$). 
I am grateful to them for inspiration and useful discussions.

\bigskip

\noindent
{\it Conventions}.
We mostly work in the complex analytic category. The general results about cohomology and Gauss-Manin connection in Section \ref{coh-GM-sec}
also hold algebraically over a field of characteristic $0$.

There are two different sign conventions for the tensor category of graded superspaces (see \cite[Sec.\ 1.2.8]{DM}). We use the Deligne's convention that
the sign appearing when swapping two symbols (that have both the degree and the parity)
is the product of the signs for the degree and for the parity. Further, we consider the de Rham differential as being even and of degree $1$.
Thus, for a pair of functions $a$, $b$ (even or odd), we have $d(ab)=da\cdot b+a\cdot db$. 
The structure even derivation $\de:\OO_X\to\om_{X/S}$ of a supercurve
in superconformal coordinates $(z,\th)$ is given by 
$$\de(f)=[dz|d\th]\cdot D(f),$$
where $D=\pa_\th+\th\pa_z$ (this differs from the convention in \cite{Deligne-letter}). Thus, we have $\de(z)=\de(\th)\cdot \th$.
The corresponding $\OO_X$-linear morphism $\varphi:\Om^1_{X/S}\to \om_{X/S}$ is given by
$$\de(df\cdot g)=\de(f)g=[dz|d\th]\cdot D(f)g.$$
In particular, $\ker(\varphi)\sub\Om^1_{X/S}$ is spanned by $dz-d\th\cdot\th=dz+\th d\th$.

\section{Supercurves of genus $1$ with an underlying odd spin structure}\label{g1-sec}

\subsection{Weierstrass functions}

Here we collect some formulas related to Weierstrass functions that we will use.
We consider the lattice $\La_\tau$ in $\C$ generated by $(1,\tau)$, where $\operatorname{Im}(\tau)>0$, and consider the corresponding
Weierstrass functions $\zeta(z,\tau)=\zeta(z,\La_{\tau})$ and $\wp(z,\tau)=-\zeta'(z,\tau)$, where we denote by prime the derivative $\partial_z$.
Recall that 
$$\zeta(z+1,\tau)=\zeta(z,\tau)+\eta_1(\tau), \ \ \zeta(z+\tau,\tau)=\zeta(z,\tau)+\eta_2(\tau),$$
where 
$$\tau\eta_1(\tau)-\eta_2(\tau)=2\pi i$$
(Legendre's relation), and
$$\eta_1(\tau)=-(2\pi i)^2\cdot \frac{E_2(q)}{12},$$
for $q=\exp(2\pi i\tau)$, where 
$$E_2(q):=1-24\sum_{n\ge 1}\frac{nq^n}{1-q^n},
$$
(see \cite[Sec.\ 1.2]{katz}\footnote{Katz uses the opposite sign in the definition of $\eta_1(\tau)$ and $\eta_2(\tau)$}).

Set 
$$\zeta_1(z,\tau):=(-2\pi i)^{-1}(\zeta(z,\tau)-\eta_1(\tau)z)$$
Then $\zeta'_1(z,\tau)=(2\pi i)^{-1}(\wp(z,\tau)+\eta_1(\tau))$, and from the Legendre's relation we get
\begin{equation}\label{zeta1-eq}
\zeta_1(z+1,\tau)=\zeta_1(z,\tau), \ \ \zeta_1(z+\tau,\tau)=\zeta_1(z,\tau)+1.
\end{equation}

Let $\a$ and $\b$ be the cycles on the elliptic curve $\C/\La_\tau$ given as images of the straight segments from $z_0$ to $z_0+1$ and from $z_0$ to $z_0+\tau$.
Then the periods of the meromorphic differentials of the second kind, $dz$ and $\zeta'_1(z,\tau)dz$, on $\C/\La_\tau$ are
$$\int_{\a} dz=1, \ \ \int_{\b} dz=\tau, \ \  \int_{\a}\zeta'_1(z,\tau)dz=0, \ \ \int_{\b}\zeta'_1(z,\tau)dz=1.$$
Hence, the normalized differentials of the second kind, 
\begin{equation}\label{e0-f0-eq}
e_0:=(1-\tau\cdot \zeta'_1(z,\tau))dz, \ \ f_0:=\zeta'_1(z,\tau)dz,
\end{equation}
give a basis of $H^1(\C/\La_\tau,\C)$, dual to $(\a,\b)$.

\subsection{Standard family of cupercurves genus $1$ with an underlying odd spin structure}\label{stand-sec}

Consider the family of supercurves $X\to S$ over the domain $S\sub\C^{1|1}$ with coordinates $(\tau,\varphi)$, defined by $\operatorname{Im}(\tau)>0$,
defined as the quotient of $\C^{1|1}\times S$ by the $\Z^2$-action
\begin{equation}\label{main-transformation}
(z,\th)\mapsto (z+1,\th), \ \ (z,\th)\mapsto (z+\tau+\th\varphi,\th+\varphi).
\end{equation}
This action preserves the vector field $D=\partial_\th+\th\partial_z$, and we equip $X$ with the corresponding superconformal structure.
For each $\tau$, we denote by $X_{\tau}$ the corresponding supercurve over $\C^{0|1}$ (with the coordinate $\varphi$).

Note that we have an NS-puncture $p:S\to X$ coming from the natural NS-puncture of $\C^{1|1}$ at the origin given by the ideal $(z,\th)$.



The family of supercurves $X\to S$ can be viewed as a homogeneous variety for a noncommutative algebraic supergroup over $S$ (see \cite[Sec.\ 3]{Rabin}).
Namely, let us consider the supergroup scheme $G$, smooth of dimension $1|1$ over $\Z$, defined as follows.
For a superring $A=A_0\oplus A_1$, the group of points $G(A)$ is the group $A_0\oplus A_1$ with the group law
$$(a_0,a_1)\cdot (a'_0,a'_1)=(a_0+a'_0+a_1a'_1,a_1+a'_1).$$
Note that $G\simeq \A^{1|1}$, and as a supergroup it is a central extension of the odd additive group by the even additive group.
Let $G_S$ denote the corresponding constant group over $S$.
Then $X$ can be identified with the quotient $G_S/(\Z^2_S)$, where the subgroup $\Z^2_S\sub G_S$ is generated by the sections
$(1,0),(\tau,\varphi)$ in $G_S(S)$. The superconformal structure on $X$ is induced by the standard superconformal structure on $\A^{1|1}$.
 
\begin{remark}
Let us give a description of an equivalent family of supercurves in terms of an arbitrary lattice, so it will be manifestly $\SL(2,\Z)$-equivariant. 
Starting with a lattice $\La\sub\C$, we consider the
corresponding Weierstrass zeta function $\zeta(z,\La)$ satisfying
$$\zeta(z+\la,\La)-\zeta(z,\La)=\eta(\la,\La)$$ 
for $\la\in \La$. Then we can replace the $\Z^2$-action \eqref{main-transformation} by the following action of $\La$:
$$(z,\th)\mapsto (z+\la+\eta(\la,\La)\cdot \th\varphi,\th+\eta(\la,\La)\cdot\varphi),$$
where $\la\in\La$. In other words, we consider the quotient $G_{\varphi}/\La$, where $G_{\varphi}=G\times \A^{0|1}_{\varphi}$, with respect
to the embedding
$$\iota_{\La,\varphi}:\La\to G_{\varphi}: \la\mapsto (\la,\eta(\la,\La)\varphi).$$
The equivalence of the above family for $\La_{\tau}=\Z+\Z\tau$ with the family corresponding to \eqref{main-transformation} is given by the coordinate change
$$\varphi=-(2\pi i)\varphi_1, \ \ z=(1+\eta(1,\La_\tau)\th\varphi)z_1, \ \ \th=\th_1+\eta(1,\La_\tau)\varphi z_1.$$

The $\SL(2,\Z)$-equivariance now easily follows by combining the identity $\eta(c\la,c\La)=c^{-1}\eta(\la,\La)$ for $c\in\C^*$ with the superconformal automorphisms
$$A_{c^{1/2}}:(z,\th)\mapsto (cz,c^{1/2}\th)$$ 
of $G=\A^{1|1}$. For example, we have $\La_{-\tau^{-1}}=\tau^{-1}\La_{\tau}$, and the commutative diagram
\begin{diagram}
\La_{\tau}&\rTo{\iota_{\La_{\tau},\varphi}}& G_{\varphi}\\
\dTo{\tau^{-1}\cdot}&&\dTo{A_{\tau^{-1/2}}}\\
\La_{-\tau^{-1}}&\rTo{\iota_{\La_{-\tau^{-1}},\tau^{-3/2}\varphi}}& G_{\varphi}
\end{diagram}
explains the modular transformation 
$$(\tau,\varphi)\mapsto (-\tau^{-1},\tau^{-3/2}\varphi)$$
(see \cite[Sec.\ 4]{FR}). The presence of a square root of $\tau$ corresponds to the fact that one should replace $\SL(2,\Z)$ by a double covering $\widehat{\SL}(2,\Z)\to\SL(2,\Z)$ 
(see \cite{Levin1}). The central subgroup $\Z/2\sub \widehat{\SL}(2,\Z)$ acts by $(\tau,\varphi)\mapsto (\tau,-\varphi)$.
\end{remark}

\subsection{Functions on the affine supercurve}\label{g1-fun-sec}

Following \cite{FR}, set
$$R=R(x,\th,\tau,\varphi):=\wp(z,\tau+\th\varphi)=\wp(z,\tau)+\th\varphi\cdot \dot{\wp}(z,\tau),$$
where $\dot{\wp}=\pa_{\tau}\wp$ (we will always use the dot to denote $\pa_{\tau}$).
Then all functions $D^nR$, for $n\ge 0$, are holomorphic on $X- p$.
Freund and Rabin  prove in \cite{FR} that for fixed $\tau$, all meromorphic functions on $X_\tau$ are rational expressions of $R$, $DR$ and $D^2R$.
Below we will describe all holomorphic functions on $X_{\tau}- p$.

For this let us introduce two more functions:
$$\Psi_1:=\th-\varphi\cdot\zeta_1(z,\tau),$$
$$\Psi_2:=\varphi\cdot \dot{\zeta}_1(z,\tau)+\th\cdot \zeta'_1(z,\tau),$$
where $\dot{\zeta}_1=\pa_{\tau}\zeta_1$, $\zeta'_1=\pa_z\zeta_1$ (the function $\Psi_1$ appears in \cite[Sec.\ 2, Eq.\ (10)]{Rabin}).

\begin{prop}\label{fun-prop}
The functions $\Psi_1$ and $\Psi_2$ are well defined holomorphic functions on $X-p$.
For each $\tau$, the functions
\begin{equation}\label{basis1-eq}
1,\Psi_1,\Psi_2,(D^nR)_{n\ge 0}
\end{equation}
form a $\C[\varphi]$-basis of $H^0(X_\tau - p,\OO)$.
One has 
\begin{equation}\label{D-Psi-eq}
D(\Psi_1)=1+\varphi\cdot\Psi_2, \ \ D(\Psi_2)=(2\pi i)^{-1}(R+\eta_1(\tau)).
\end{equation}
\end{prop}

\begin{proof}
The invariance of $\Psi_1$ and $\Psi_2$ under the transformations \eqref{main-transformation} follow from the relations \eqref{zeta1-eq},
and from the following relation obtained from \eqref{zeta1-eq} by differentiating with respect to $\tau$:
$$\dot{\zeta}_1(z+\tau,\tau)-\dot{\zeta}_1(z,\tau)=-\zeta'_1(z,\tau).
$$

Let $f$ be a meromorphic function on $X_\tau$ invariant under \eqref{main-transformation}. We can write $f=A+\th B$, where $A=A_0(z)+\varphi A_1(z)$, $B=B_0(z)+\varphi B_1(z)$.
Then setting $\De_\tau g(z):=g(z+\tau)-g(z)$, we get the following equations on $A_i,B_i$:
\begin{align*}
&A_i(z+1)=A_i(z), \ \ B_i(z+1)=B_i(z), \ i=1,2, \\
&\De_\tau A_0=0, \ \ \De_{\tau} B_0=0, \\
&\De_{\tau} A_1(z)=-B_0(z), \ \ \De_\tau B_1(z)=-A'_0(z).
\end{align*}

Assume first that $f$ has pole of order at most $1$ at $z=0$. Then these equations imply that $A_0$ and $B_0$ are constants, hence, $B_1$ is also constant.
We also get that $A_1$ is of the form $-B_0\zeta_1+c$, where $c\in\C$. Thus, we get
$$f=A_0+\varphi(-B_0\zeta_1+c)+\th(B_0+\varphi B_1)=(A_0+\varphi c)-(B_0+\varphi B_1)\Psi_1,
$$
i.e., $f$ is a $\C[\varphi]$-linear combination of $1$ and $\Psi_1$.

Next, we observe that the functions \eqref{basis1-eq} have expansions near $z=0$,
\begin{align*}
&\Psi_1=\th+(2\pi i)^{-1}\cdot\frac{\varphi}{z}+O(z), \ \ \Psi_2=(2\pi i)\cdot \frac{\th}{z^2}+O(1), \\
&D^{2n}R=(-1)^n\frac{(n+1)!}{z^{n+2}}+O(1), \ \ D^{2n+1}R=(-1)^{n+1} \frac{(n+2)!\th}{z^{n+3}}+O(1), \ n\ge 0.
\end{align*}
These expansions immediately imply that \eqref{basis1-eq} are linearly independent over $\C[\varphi]$.
Also, we see that for any holomorphic function $f$ on $X-p$ we can find a $\C[\varphi]$-linear combination $g$ of $(\Psi_2,D^nR)$ such that
$f-g$ has a pole of order $\le 1$ at $z=0$. Applying the particular case considered above, we deduce that $f$ is a $\C[\varphi]$-linear combination of
the elements \eqref{basis1-eq}.

The equations \eqref{D-Psi-eq} are obtained by calculating the expansions of both sides at $z=0$.
\end{proof}

Note that global holomorphic functions on $X_\tau$ are $\C$-linear combinations of $1$, $\varphi$ and $\varphi\Psi_1=\varphi\th$,
so $H^0(X_\tau,\OO)$ is not free as a $\C[\varphi]$-module.

\subsection{An algebraic model}

Here we explain a simple way to construct the family $X/S$ algebraically
(other constructions can be found in \cite{FR} and \cite{Levin}).

It is convenient to replace $D$ by another global odd derivation, generating the same distribution:
$$\wt{D}:=(2\pi i)^{-1/2}[1-(2\pi i)\cdot\frac{E_2}{12}\th\varphi]D$$
(we choose once and for all a square root $(2\pi i)^{1/2}$).
Let us set
$$x:=(2\pi i)^{-2}[1-(2\pi i)\frac{E_2}{6}\th\varphi]\cdot R, \ \ y:=(2\pi i)^{-3}[1-(2\pi i)\frac{E_2}{4}\th\varphi]\cdot D^2R,$$
$$\psi:=(2\pi i)^{1/2}\Psi_1, \ \ \wt{\varphi}:=(2\pi i)^{1/2}\varphi, \ \ \wt{\Psi}_2:=(2\pi i)^{-1/2}\Psi_2.$$

We will also use the following relations:
$$g_2(\tau)=(2\pi i)^4\cdot\frac{E_4(q)}{12}, \ \ g_3(\tau)=-(2\pi i)^6\cdot \frac{E_6(q)}{216}, \ \ \eta_1(\tau)=-(2\pi i)^2\cdot \frac{E_2(q)}{12},$$
$$\dot{g}_2(\tau)=(2\pi i)^5\cdot\frac{\partial E_4(q)}{12}, \ \ \dot{g}_3(\tau)=-(2\pi i)^7\cdot \frac{\partial E_6(q)}{216}, \ \ \dot{\eta}_1(\tau)=-(2\pi i)^3\cdot \frac{\partial E_2(q)}{12}$$
where $q=\exp(2\pi i\tau)$, 
$\partial=q\partial_q$, $\partial E_2=(E_2^2-E_4)/12$, $\partial E_4=(E_2E_4-E_6)/3$, and $\partial E_6=(E_2E_6-E_4^2)/2$.

\begin{lemma}
One has the following relations
\begin{equation}\label{Psi2-eq}
\wt{\Psi}_2=x\psi+\frac{1}{2}\wt{\varphi}y-\frac{E_2}{12}\psi,
\end{equation}
\begin{equation}\label{DR-eq}
\wt{D}(x)=y\psi+\wt{\varphi}\cdot[2x^2-\frac{E_4}{36}],
\end{equation}
\begin{equation}\label{D3R-eq}
\wt{D}(y)=(6x^2-\frac{E_4}{24})\psi+3\wt{\varphi}xy,
\end{equation}
\begin{equation}\label{Dpsi-eq}
\wt{D}(\psi)=1+\wt{\varphi}\psi x,
\end{equation}
\begin{equation}\label{main-eq}
y^2-4x^3+\frac{1}{12}(E_4-\frac{E_6}{3} \psi\wt{\varphi})x-\frac{1}{216}(E_6-\frac{E_4^2}{2} \psi\wt{\varphi})=0.
\end{equation}
\end{lemma}

\begin{proof}
These relations are checked by considering expansions of both sides in $(z,\th)$ near $z=0$, and using Proposition \ref{fun-prop}.
The last equation can also be checked by substituting $\tau\mapsto \tau+\th\varphi$ in the standard Weierstrass cubic relation for $\wp'$ and $\wp$.
\end{proof}

Let $C_{\tau,\varphi}$ be the cubic in the relative superprojective space $\P^{2|1}_S$ with homogeneous coordinates $(x_0:x_1:x_2:\phi)$
given by the homogenization of the equation \eqref{main-eq}:
$$x_0x_2^2-4x_1^3+\frac{1}{12}(E_4x_0^2-\frac{E_6}{3} x_0\phi\wt{\varphi})x_1-\frac{1}{216}(E_6x_0^3-\frac{E_4^2}{2}\cdot x_0^2\phi\wt{\varphi})=0.$$

\begin{theorem}\label{algebra-thm}
(i) The algebra $\OO(X_\tau-p)$ of holomorphic functions on $X_\tau-p$ (resp., the algebra $\OO(X-p)$) is generated over $\C[\varphi]$ (resp., $\OO(S)$) by
$\psi$ (an odd generator), $x$ and $y$ (even generators), with the defining relation \eqref{main-eq}. We have
the following $\C[\varphi]$-basis of $\OO(X_\tau-p)$:
\begin{equation}\label{basis2-eq}
(x^n, yx^n, x^n\psi, yx^n\psi)_{n\ge 0}.
\end{equation}
The odd derivation $\wt{D}$ acts on $\OO(X_\tau-p)$ by the formulas \eqref{DR-eq}, \eqref{D3R-eq} and \eqref{Dpsi-eq}.

\noindent
(ii) The map $(z,\th)\mapsto (1:x(z,\th):y(z,\th):\psi(z,\th))$ extends to a regular morphism $f:X_\tau\to C_{\tau,\varphi}$, which identifies $X_{\tau}$ with
the blow-up $\wt{C}_{\tau,\varphi}$ of $C_{\tau,\varphi}$ at the ideal $(\frac{\phi}{x_2}+2\wt{\varphi}\frac{x_1^2}{x_2^2}, \frac{x_0}{x_2})$ supported
at the point $x_0=x_1=0$.
\end{theorem}

\begin{proof}
(i) We can compute the expansions of the elements \eqref{basis2-eq} at $z=0$:
\begin{align*}
& x^n=\frac{a_n}{z^{2n}}+\ldots, & yx^n=\frac{b_n}{z^{2n+3}}+\ldots, \\ 
& x^n\psi=\frac{c_n\th}{z^{2n}}+\ldots (\mod \varphi), & yx^n\psi=\frac{d_n\th}{z^{2n+3}}+\ldots (\mod\varphi),
\end{align*}
where $a_n,b_n,c_n,d_n$ are nonzero complex constants.
It follows that for any element $f\in \OO(X_\tau-p)$, there exists a $\C[\varphi]$-linear combination $g$ of these elements
such that $f-g$ has pole of order at most $1$ at $z=0$. As we have seen in the proof of Proposition \ref{fun-prop}, this
implies that $f-g$ is a linear combination of $1$ and $\psi$. Thus, the elements \eqref{basis2-eq} span $\OO(X_\tau-p)$.
The above expansions also show linear independence of \eqref{basis2-eq} over $\C[\varphi]$.

Let $F(x,y,\psi)$ denote the left-hand side of \eqref{main-eq}. Then we have a surjection
\begin{equation}\label{quotient-presentation}
\C[\varphi][x,y,\psi]/(F)\to \OO(X_\tau-p)
\end{equation}
It is easy to see that the images of \eqref{basis2-eq} span $\C[\varphi][x,y,\psi]/(F)$ over $\C[\varphi]$. Since their images in
$\OO(X_\tau-p)$ are linearly independent, we deduce that \eqref{quotient-presentation} is an isomorphism.


\noindent
(ii) It is clear that the map $f$ is regular away from $z=0$. Furthermore, part (i) shows that the restriction $f$ to $X_\tau-p$ induces an isomorphism
of $X_\tau-p$ with the affine part $(x_0\neq 0)$ in $C_{\tau,\varphi}$.

On the other hand, near $z=0$, the expansions
$$f^*\frac{x_0}{x_2}=\frac{1}{y}=-\frac{(2\pi i)^3}{2}z^3(1+O(z^4))(1-(2\pi i)\frac{E_4}{4}\th\varphi), \ \ 
f^*\frac{x_1}{x_2}=\frac{x}{y}=-\frac{2\pi i}{2}z(1+O(z^4))(1-(2\pi i)\frac{E_2}{12}\th\varphi),$$
$$f^*\frac{\phi}{x_2}=\frac{\psi}{y}=-\frac{(2\pi i)^2}{2}\wt{\varphi}z^2-\frac{(2\pi i)^{7/2}}{2}\th z^3+O(z^4).
$$
show that $f$ is regular near $z=0$. 

The blow up $\wt{C}_{\tau,\varphi}\to C_{\tau,\varphi}$ at the ideal $(\phi':=\frac{\phi}{x_2}+2\wt{\varphi}\frac{x_1^2}{x_2^2},\frac{x_0}{x_2})$, we add
to the local ring of $C_{\tau,\varphi}$ the function $\phi'\cdot (x_0/x_2)^{-1}$.
We have
$$f^*(\phi')=-\frac{(2\pi i)^{7/2}}{2}\th z^3+O(z^4).$$
$$f^*(\phi'\cdot \frac{x_2}{x_0})=(2\pi i)^{1/2}\th+O(z).$$
Since $f^*(x_1/x_2)=-2\pi i z(1+O(z^2))(1-(2\pi i)\frac{E_2}{12}\th\varphi)$,
this implies the identification of $X_\tau$ with the blow up $\wt{C}_{\tau,\varphi}$.
\end{proof}

Let $Q\sub \OO(S)$ denote the (quasimodular) $\C$-subalgebra generated by $E_2$, $E_4$, $E_6$ and $\varphi$. Let
$A\sub \OO(X-p)$ denote the $Q$-subalgebra generated by $x$, $y$ and $\psi$.

\begin{cor}\label{quasimod-cor}
The pair $(A,\wt{D})$ is generated as an algebra with a derivation over $Q$ by $\psi$ and $x$ (or by $\psi$ and $\wt{\Psi}_2$).
\end{cor}

\begin{proof} 
Since $\th\varphi=\psi\varphi$, we can replace $\wt{D}$ by $D$, $x$ by $R$ and $y$ by $D^2R$. Now the generation by $\psi$ and $x$ becomes clear, while
the generation by $\psi$ and $\wt{\Psi}_2$ follows from the second formula in \eqref{D-Psi-eq}.
\end{proof}


\subsection{Kodaira-Spencer map and compactification}

Let us set $S'=S/\Z$, where we take the quotient by the $\Z$-action $\tau\mapsto \tau+1$. Then 
we have $S'=\Disk\times\C^{0|1}$, where $\Disk$ is the punctured unit disk $0<|q|<1$ with the coordinate 
$q=\exp(2\pi i\tau)$.
Clearly, our family of supercurves with an NS-puncture $(X\to S,p)$ descends to a family $(X'\to S',p')$ over $S'$.
We want to extend it to a family of stable supercurves over $\ov{S}=\ov{\Disk}\times\C^{0|1}$, where $\ov{\Disk}$ is the unit disk $|q|<1$.
For this we note that the coefficients of the cubic $C_{\tau,\varphi}$ depend on $\tau$ via $E_4$ and $E_6$, so it descends to a cubic
$C_{q,\varphi}\sub \P^{2|1}_{\ov{S}}$. 

\begin{theorem}\label{stable-thm} 
Let us define the family $\ov{X}\to \ov{S}$ with a section $\ov{p}\sub \ov{X}$ 
as the blow up $\wt{C}_{q,\varphi}$ of the cubic $C_{q,\varphi}$ at the ideal $(\frac{\phi}{x_2}+2\wt{\varphi}\frac{x_1^2}{x_2^2}, \frac{x_0}{x_2})$ supported
at the point $x_0=x_1=0$.
Then it is a family of stable supercurves of genus $1$ with an NS-puncture, with the trivial relative Berezinian $\om_{\ov{X}/\ov{S}}$,
and the corresponding morphism to the moduli space,
$$\ov{S}\to \ov{\SS}^{odd}_{1,1}$$
induces an isomorphism on tangent spaces.
\end{theorem}

\begin{proof}
{\bf Step 1.}
First, let us consider the affine cubic $C^a_{q,\varphi}:=C_{q,\varphi}-\ov{p}\sub \A^{2|1}_{\ov{S}}$. Note that $C^a_{q,\varphi}$ is smooth over $\ov{S}$ everywhere except at one point in
the fiber over $q=0$. We claim that the odd derivation $\wt{D}$ given by \eqref{DR-eq}, \eqref{D3R-eq} and \eqref{Dpsi-eq} gives
a superconformal structure on the smooth locus of $C^a_{q,\varphi}\to \ov{S}$. Indeed, we have to check that the vector fields $\wt{D}$ and $\wt{D}^2$ are linearly
independent at all smooth points. Note that we have the following relation in $\Om^1_{C^a_{q,\varphi}/\ov{S}}$ modulo $\varphi$:
$$2ydy-12 x^2dx+\frac{E_4}{12}dx\equiv 0 \mod (\varphi).$$
First, we consider the locus $y\neq 0$. The above relation shows that $dx$ and $d\psi$ form a basis of $\Om^1_{C^a_{q,\varphi}/\ov{S}}$ on this locus. Hence, our claim on this locus follows from
the invertibility of
$$\left(\begin{matrix} \wt{D}(\psi) & \wt{D}(x) \\ \wt{D}^2(\psi) & \wt{D}^2(x)\end{matrix}\right)\equiv 
\left(\begin{matrix} 1 & y\psi \\ 0 & y\end{matrix}\right).$$
Next, we consider the locus $y=0$, $12x^2-\frac{E_4}{12}\neq 0$. Then we can take $dy$ and $d\psi$ as a basis of $\Om^1_{C^a_{q,\varphi}/\ov{S}}$, and our claim followss from the invertibility of
$$\left(\begin{matrix} \wt{D}(\psi) & \wt{D}(y) \\ \wt{D}^2(\psi) & \wt{D}^2(y)\end{matrix}\right)\equiv 
\left(\begin{matrix} 1 & \wt{D}(y) \\ 0 & 6x^2-\frac{E_4}{24}\end{matrix}\right).$$

\noindent
{\bf Step 2.}
Let $(z,\th)$ denote the formal relative superconformal coordinates on $C_{q,\varphi}$ at infinity, induced by the coordinates on $\C^{1|1}$ from the quotient construction. 
Note that the functions $(x/y,\psi+2\wt{\varphi}x^2/y)$ define local relative coordinates on $\wt{C}_{q,\varphi}$ along $p$, such that the change of variables from $(z,\th)$ to
$(x/y,\psi+2\wt{\varphi}x^2/y)$ is given by formal series of the form
$$x/y=-2\pi iz+\ldots, \ \ \psi+2\wt{\varphi}x^2/y=(2\pi i)^{1/2}\th+(2\pi i)\wt{\varphi}\frac{E_2}{12}z+\ldots$$
where the coefficients of the higher order terms are polynomials in $E_2$, $E_4$ and $E_6$. Since $E_2$, $E_4$ and $E_6$ are regular at $q=0$,
these formulas constitute an invertible change of variables everywhere along $\ov{p}$, hence $(z,\th)$ extend to formal relative coordinates at $\ov{p}\sub \wt{C}_{q,\varphi}$.
Thus, we can view $\wt{C}_{q,\varphi}$ as glued from $C_{q,\varphi}$ and a relative formal superdisk at infinity with coordinates $(z,\th)$. 
Note that the superconformal structure on the smooth part of $C_{q,\varphi}$ agrees with the superconformal structure given by $D=\partial_\th+\th\partial_z$ on the formal
superdisk, hence, they glue into a superconformal structure on the smooth part of $\wt{C}_{q,\varphi}$.

\noindent
{\bf Step 3.}
We claim that 
$$\ov{s}:=\de(\psi)(1-x\wt{\varphi}\psi)$$ 
is a trivialization of the relative Berezinian of the affine cubic over $S$, which extends to a trivialization of the relative dualizing
sheaf $\om_{\ov{X}/\ov{S}}$. Indeed, first we observe that $\om_{\ov{X}/\ov{S}}$ is a line bundle. Indeed, it is clear on the smooth part, while on the affine part $C^a_{q,\varphi}$
it follows from the fact that it is a divisor in $\A^{2|1}$.

Next, on the affine part over $S$, $\ov{s}$ differs by an invertible function from $\de(\th)$, so it is a generator. Also, near the infinite point we have
$$\ov{s}=(2\pi i)^{1/2}\cdot \de(\th)(1+(2\pi i)\frac{E_2}{12}\th\varphi),$$
which is a generator everywhere near $\ov{p}$ (since $E_2$ is regular at $q=0$).
The above formula also shows that $\ov{s}$ is a generator near infinity on the fiber over $q=0$. Hence, $\ov{s}$ can have zeros only in codimension $\ge 2$.
Note that the family $\ov{X}\to \ov{S}$ is the standard family of stable curves over $\{q\in\C \ |\ 0<|q|<1\}$, in particular, $\ov{X}_{\bos}$ is smooth.
Hence, $\ov{s}|_{\ov{X}_{\bos}}$ is non-vanishing everywhere, and so is $\ov{s}$.

\noindent
{\bf Step 4.}
We already equipped the smooth part of $\wt{C}_{q,\varphi}$ with a superconformal structure. Next, we need to check that the corresponding derivation
$\de:\OO_{\ov{X}}\to \om_{\ov{X}/\ov{S}}$ is regular near the singular point of the fiber over $q=0$ and induces an isomorphism of the odd components.
For any function $f$ on the smooth part of $C^a_{q,\varphi}$, we have
$$\de(f)=\de(\th)\cdot D(f)=(2\pi i)^{-1/2}\ov{s}\cdot (1-(2\pi i)\frac{E_2}{12}\th\varphi)D(f)=\ov{s}\cdot\wt{D}(f).$$
Hence, the regularity follows from the formulas \eqref{DR-eq}, \eqref{D3R-eq} and \eqref{Dpsi-eq} for $\wt{D}$.
Also, $\de(\psi)$ generates $\om_{C^a_{q,\varphi}}$ near the singular point, so $\varphi$ induces an isomorphism of the odd components. 

Since the corresponding bosonic family is a family of stable curves of genus $1$ with $1$ marked point, we obtain that $(\ov{X}\to \ov{S}, \ov{p})$
is a family of stable supercurves of genus $1$ with one NS puncture.

\noindent
{\bf Step 5.} Computation of the Kodaira-Spencer map on $S$.

Recall that for a family $\pi:X\to S$ of smooth supercurves with an NS puncture $p\sub X$, the Kodaira-Spencer map (KS map) 
is associated with the exact sequence
$$0\to \AA_{(X,p)/S}\to \AA_{X,p}\to \pi^{-1}\TT_S\to 0,$$
where $\AA_{X,p}$ is the sheaf of vector fields on $X$ preserving the structure distribution $\DD_{X/S}$ 
and the puncture $p$, $\AA_{(X,p)}$ is the similar subsheaf of the relative tangent
bundle. Looking at the connecting homomorphism on the cohomology, we get for each $s\in S$, a map
$$KS:T_sS\to H^1(X_s,\AA_{X_s,p}),$$
 
Note that the fact that the KS map is an isomorphism on the even components follows from the similar fact about the bosonization of our family.
It is still instructive to compute the KS map fully.

To compute the KS map we will use the covering of $X$ by $U_1=X-p$ and $U_2$, the formal neighborhood of $p$. For a given vector field $v$ on the base,
$KS(v_s)$ is given as the Cech cohomology class of the difference $\wt{v}_1-\wt{v}_2$, where $\wt{v}_i$ are the lifts of $v$ to $\AA_{X,p}(U_i)$, for $i=1,2$.
More precisely, we additionally use an isomorphism of sheaves 
$$\AA_{(X,p)/S}\simeq \TT_X/\DD_{X/S}(-D)\simeq \om_{X/S}^{-2}(-D)$$
(see \cite[Prop.\ 3.9]{FKP-stable}).
Furthermore, for $s\in S$, we have natural isomorphisms
$$\AA_{X_s}^+\simeq \om_{C_s}^{-1}\simeq \OO_{C_s}, \ \ (\om_{X_s}^{-2})^-\simeq L\simeq \OO_{C_s},$$
where $(C_s,L)$ is the underlying elliptic curve with the trivial spin structure $L\simeq \OO$. 
One can check that the corresponding trivializing sections of $\TT_X/\DD_{X/S}$ are $D^2$ and $\th D^2$.
These isomorphisms induce isomorphisms
$$H^1(\AA_{X_s,p})^+\simeq H^1(C_s,\OO_{C_s}(-p)), \ \ H^1(\AA_{X_s,p})^-\simeq H^1(C_s,\OO_{C_s}(-p)).$$

Thus, to compute $KS(\partial_\tau)$, we need to find a lifting of $\partial_\tau$ to a vector field on $U_1=X-p$ preserving the structure distribution
(over $U_2$ the obvious lift $\partial_\tau$ is in $\AA_{X,p}(U_2)$). One can check that such a lift is provided by
\begin{equation}\label{Dtau-eq}
D_{\tau}:=\partial_{\tau}+(\zeta_1+\th\varphi\dot{\zeta}_1)D^2+\frac{1}{2}\Psi_2\cdot D=\partial_\tau+(\zeta_1+\frac{1}{2}\th\varphi\dot{\zeta}_1)\partial_z+
\frac{1}{2}(\th\zeta'_1+\varphi\dot{\zeta}_1)\partial_\th.
\end{equation}
Thus, the corresponding cocycle with values in superconformal vector fields is represented by $(\zeta_1+\th\varphi\dot{\zeta}_1)D^2+\frac{1}{2}\Psi_2\cdot D$ on $U_{12}$.
Hence, we get that $KS(\partial_{\tau})$ is the class of $\zeta_1(z,\tau)=(2\pi i)^{-1}\frac{1}{z}+\ldots$ in $H^1(C_s,\OO(-p))$,
which is a generator of the latter group (note that when considering $X_s$ we set $\varphi=0$).

Similarly, as a lift of $\partial_\varphi$ to $U_1$ we can take
\begin{equation}\label{Deta-eq}
D_{\varphi}:=\partial_{\varphi}+[\varphi(\zeta_1^2-2\th \zeta_1)]D^2+[\zeta_1-\th\varphi\zeta_1\zeta'_1)]D=\partial_\varphi+(\varphi\zeta_1^2-\th\zeta_1)\partial_z+(\zeta_1-\th\varphi\zeta_1\zeta'_1)
\partial_\theta,
\end{equation}
while on $U_2$ we take $\partial_\varphi$. Hence $KS(\partial_{\varphi})$ is the class of $-2\zeta_1(z,\tau)$ in $H^1(C_s,\OO(-p))$.
This proves that the KS map is an isomorphism for any $s\in S$.

\noindent
{\bf Step 6.} Computation of the Kodaira-Spencer map on $\ov{S}$.

As shown in \cite[Sec.\ 6]{FKP-stable}, there is a version of the KS map for families of stable supercurves, 
$$T_{\ov{S},\ov{S}_0}|_{s_0}\to H^1(X_{s_0},\AA_{X_{s_0},p}),$$
where $\ov{S}_0\sub \ov{S}$ is the divisor of nodal curves and $s_0\in\ov{S}_0$. 
For the base of our family, near $q=0$ the basis of $T_{S,S_0}$ is given by $(2\pi i)q\partial_q=\partial_\tau$ and $\partial_\varphi$.
The fiber of our family over $q=0$ (and $\varphi=0$) is the split supercurve $C_0\times \A^{0|1}$ with the odd coordinate $\th=\psi$.
The formulas \eqref{DR-eq}, \eqref{D3R-eq} and \eqref{Dpsi-eq} show that the vector field $D$ extends to a global relative vector field on $\ov{X}\to \ov{S}$ and it is easy to check that
$D^2$ (resp., $\th D^2$) generates $\AA_{X_{s_0}}^+$ (resp., $\AA_{X_{s_0}}^-$) as an $\OO_{C_0}$-module.

To compute the KS map, first, we need to check that the vector fields $D_\tau$ and $D_\varphi$ given by \eqref{Dtau-eq} and \eqref{Deta-eq}, preserve the $Q$-subalgebra $A\sub \OO(X-p)$ (see Corollary \ref{quasimod-cor}) and hence, extend to derivations on $\ov{X}$.
We observe that we have the following relations
$$D_\tau D-D D_\tau=-\frac{1}{2}D(\Psi_2) D,$$
$$D_\varphi D+D D_{\varphi}=\zeta'_1\psi_1 D=(2\pi i)^{-1}(x-(2\pi i)^2\frac{E_2}{12})\psi_1 D.$$
Since $D(\Psi_2)\in A$ (see \eqref{D-Psi-eq}), it remains to prove that $D_\tau(\psi)$, $D_\tau(\Psi_2)$, $D_\varphi(\psi)$ and $D_\varphi(\Psi_2)$ are in $A$.
A computation using the relations
$$\dot{\zeta}_1+\zeta_1\zeta'_1=(2\pi i)^{-2}\frac{1}{2}\wp', \ \ \dot{\wp}+\zeta_1\wp'=(2\pi i)^{-1}(2\wp^2+(2\pi i)^{2}\frac{E_2}{6}-\frac{g_2}{3}), \ \
\dot{\wp}'+\zeta_1\wp''=(2\pi i)^{-1}3(\wp+(2\pi i)^2\frac{E_2}{12})\wp',$$                                                                        
gives
\begin{align*}
&D_\tau(\psi)=\frac{1}{2}\Psi_2-(2\pi i)^{-1}\varphi\frac{1}{2}D^2R, \\
&D_\tau(\Psi_2)=(2\pi i)^{-2}\frac{1}{2}D^3R-(2\pi i)^{-1}\frac{1}{2}(\wp-(2\pi i)^2\frac{E_2}{12})\Psi_2, \\
&D_\varphi(\psi)=0, \\
&D_\varphi(\Psi_2)=(2\pi i)^{-2}\frac{1}{2}D^2R[1-(2\pi i)^{-1}3\psi\varphi(R+(2\pi i)^2\frac{E_2}{12})].
\end{align*}

Thus, the values of the KS map on $D_\tau$ and $D_\varphi$ are still given by the same Cech cocycles as above.
It is well known that for the nodal cubic $C_0$, the class of $\frac{1}{z}$ is still a generator of $H^1(C_0,\OO(-p))$, which shows that the KS map
is still an isomorphism for the fiber over $q=0$. 
\end{proof}

\section{Cohomology and Gauss-Manin connection}\label{coh-GM-sec}

\subsection{Cohomology and differentials}


Let $\pi:X\to S$ be a smooth proper supercurve over a smooth base $S$. 
We work either in the category of schemes over a field $k$ of characteristic zero and then assume that $S$ is affine, or
in the holomorphic category and assume that $S$ is Stein.

Following \cite{Deligne-letter}, we use the identification of the structure distribution $\DD\sub \TT_{X/S}$ with $\om_{X/S}^{-1}$, and consider the dual
map $\de:\Om^1_{X/S}\to \om_{X/S}$, which induces a chain map of complexes
\begin{equation}\label{Om-om-map}
\Om^{\bullet}_{X/S}\to [\OO_X\rTo{\de}\om_{X/S}],
\end{equation}
where $\Om^{\bullet}_{X/S}$ is the relative de Rham complex.
It is well known that this map is a quasi-isomorphism in the analytic category (since both complexes are resolutions of the relative constant sheaf $\C_{X/S}$.
The following result gives a stronger algebraic verion of this statement.

\begin{prop}\label{de-rham-quasi-isom-prop}
The chain map \eqref{Om-om-map} fits into a short exact sequence
$$0\to K^{\bullet}_{X/S}\to \Om^\bullet_{X/S}\to [\OO_X\rTo{\de} \om_{X/S}]\to 0,$$
where the complex $K^\bullet_{X/S}$ is canonically contractible, i.e., there is a canonical homotopy between $0$ and identity maps
\end{prop}

\begin{proof}
It is known (see \cite{Deligne-letter}) that we have an exact sequence of vector bundles on $X$,
\begin{equation}\label{om2-om-seq}
0\to \om^2_{X/S}\to \Om^1_{X/S}\rTo{\de} \om_{X/S}\to 0,
\end{equation}
where $\om^2_{X/S}$ has rank $1|0$ and $\om_{X/S}$ has rank $0|1$.
Thus, $K^\bullet_{X/S}$ is the following subcomplex of $\Om^\bullet_{X/S}$:
$$0\to \om^2_{X/S}\to \Om^2_{X/S}\to \Om^3_{X/S}\to \ldots.$$

Now we observe that by passing to the exterior powers,
the exact sequence \eqref{om2-om-seq} induces exact sequences
$$0\to \om^{i+1}_{X/S}\to \Om^i_{X/S}\rTo{\wedge^i(\de)} \om^i_{X/S}\to 0$$
for $i\ge 1$.
The obtained maps $\Om^2_{X/S}\to \om^2_{X/S}$, and
$$\Om^i_{X/S}\to \om^i_{X/S}\to \Om^{i-1}_{X/S}$$
for $i>2$, give the required homotopy.
\end{proof}

Thus, the de Rham cohomology of $X/S$ can be calculated as 
$$\HH^i_{X/S}:=R^i\pi_*[\OO_X\to \om_{X/S}].$$
These are locally free $\OO_S$-modules of finite rank (by the general theory of $D$-modules on supervarieties developed in \cite{Penkov}, this follows from the existence of a flat connection on $\HH^i_{X/S}$, see Sec.\ \ref{GM-sec} below). 

Let $i:S_0\hra S$ denote the bosonization of $S$, and let $\pi_0:C\to S_0$ denote the family of usual curves
underlying $\pi:X\to S$. Then the pull back induces a natural isomorphism 
$$i^*R\pi_*(\Om^\bullet_{X/S})\rTo{\sim} R\pi_{0*}(\Om^\bullet_{C/S_0})$$
(see \cite{P-super-derham}). Furthermore, since the cohomology sheaves $\HH^n_{X/S}$ are locally free, in fact, we have
\begin{equation}\label{bos-H-iso}
i^*\HH^n_{X/S}\simeq \HH^n_{C/S_0}.
\end{equation}

Assume now that we have a section $p\sub X$ of the projection $X\to S$.
Consider the flat covering $X=U_1\cup U_2$, where $U_1=X-p$, and $U_2$ is a formal neighborhood of $p$.
Let $\CC[\OO_X\to \om_{X/S}]$ be the corresponding relative Cech complex.

\begin{prop}\label{H1-prop}
The natural projection $\CC^1[\OO_X\to\om_{X/S}]\to \om_{X/S}(U_1)$ induces an isomorphism
$$\HH^1_{X/S}=R^1\pi_*[\OO_X\to \om_X]\rTo{\sim} \coker(\de:\OO_X(X-p)\to \om_{X/S}(X-p)).$$
\end{prop}

\begin{proof} The proof is similar to the classical case (cf.\ \cite[Lemma 5.1]{P-bu} for a proof of a similar statement for the separating node degeneration of curves). 
Namely, taking the quotient of the Cech complex of $[\OO_X\to \om_{X/S}]$ by an exact subcomplex, we get the quasi-isomorphic complex
$$\OO(U_1)\to \om_{X/S}(U_1)\to \om_{X/S}(U_{12})/(\de(\OO(U_{12})+\om_{X/S}(U_2)).$$
It remains to check that the image of
the restriction map $\om_{X/S}(U_1)\to \om_{X/S}(U_{12})$ is contained in
$\de(\OO(U_{12}))+\om_{X/S}(U_2)$. This can be checked as follows.
It is easy to check using local formal relative coordinates that  
$$\om_{X/S}(U_{12})/(\de(\OO(U_{12}))+\om_{X/S}(U_2))\simeq \OO_S$$
(the isomorphism is given by the residue, see \cite[Sec.\ 2.8]{FKP}).
On the other hand, 
we have an isomorphism
$$\om_{X/S}(U_{12})/(\de(\OO(U_{12}))+\om_{X/S}(U_1)+\om_{X/S}(U_2))\simeq \HH^2_{X/S},$$
which is a locally free $\OO_S$-module of rank $1$.
Since the natural projection
$$\OO_S\simeq\om_{X/S}(U_{12})/(\de(\OO(U_{12}))+\om_{X/S}(U_2))\rTo{pr} \om_{X/S}(U_{12})/(\de(\OO(U_{12}))+\om_{X/S}(U_1)+\om_{X/S}(U_2))\simeq \HH^2_{X/S},$$
is surjective, it is an isomorphism. This implies that the map $pr$ is an isomorphism, which implies the assertion.
\end{proof}

\subsection{Gauss-Manin connection: generalities}\label{GM-sec}

Recall that the Gauss-Manin connection $\nabla^{GM}$ on the de Rham cohomology $R^i\pi_*(\Om^\bullet_{X/S})$ arises as a connecting homomorphism associated with
the exact sequence of complexes 
$$0\to \Om^{\bullet-1}_{X/S}\ot \pi^*\Om^1_S\to \Om^\bullet_X/(\Om^{\bullet-2}_X\cdot \pi^*\Om^2_S)\to \Om^\bullet_{X/S}\to 0$$
(see \cite{P-super-derham}). 

It is easy to see that isomorphisms \eqref{bos-H-iso} are
compatible with the Gauss-Manin connections.

\begin{prop}\label{GM-recipe-prop} Let $p\sub X$ be a section of $X\to S$.
Under the identification $\HH^1_{X/S}\simeq \coker(\de:\OO_X(X-p)\to \om_{X/S}(X-p))$ (see Proposition \ref{H1-prop}), the Gauss-Manin connection
$$\nabla^{GM}:\HH^1\to \HH^1\ot \Om^1_S$$
can be computed as follows. Start with a representative $\Phi\in \om_{X/S}(U_1)$. Then there exists a unique
$\om\in \Om^1_{X/S}(U_1)$ such that $\de(\om)=\Phi$ and $d_{X/S}\om=0$.
Lift $\om$ to some $\wt{\om}\in \Om^1_X(U_1)$ and consider its total differential $d\wt{\om}\in \Om^2_X(U_1)$.
Then $d\wt{\om}\mod \Om^2_S$ is in the image of $\Om^1_{X/S}(U_1)\ot \Om^1_S$ and we have
$$\nabla^{GM}([\Phi])=[(\de\ot \id)(d\wt{\om}\mod \Om^2_S)],$$
where on the right we consider the map $\de\ot\id: \Om^1_{X/S}(U_1)\ot \Om^1_S\to \om_{X/S}(U_1)\ot\Om^1_S$. 
\end{prop}

\begin{proof}
By Proposition \ref{de-rham-quasi-isom-prop}, the map $\de:\Om^1_{X/S}(U_1)\to \om_{X/S}(U_1)$ is surjective and its kernel maps by $d_{X/S}$ isomorphically to 
$\ker(d:\Om^2_{X/S}(U_1)\to \Om^3_{X/S}(U_1))$. This implies that every $\Phi\in \om_{X/S}(U_1)$ lifts uniquely to
$\om\in \Om^1_{X/S}(U_1)$ such that $d_{X/S}(\om)=0$.

By Proposition \ref{H1-prop}, the class of $\Phi$ in $\HH^1_{X/S}$ corresponds to a Cech cocycle
$$(f,\Phi,\a)\in \OO(U_{12})\oplus \om_{X/S}(U_1)\oplus \om_{X/S}(U_2)=\CC^1[\OO_X\to \om_{X/S}],$$
where 
$\de(f)=-\Phi|_{U_{12}}+\a|_{U_{12}}$.
By Proposition \ref{de-rham-quasi-isom-prop}, it can be lifted to a Cech cocyle
$$(f,\om_1,\om_2)\in \OO(U_{12})\oplus \Om^1_{X/S}(U_1)\oplus \Om^1_{X/S}(U_2)=\CC^1(\Om^\bullet_{X/S}),$$
where 
$$d_{X/S}f=-\om_1|_{U_{12}}+\om_2|_{U_{12}}, \ d_{X/S}\om_1=d_{X/S}\om_2=0, \ \de(\om_1)=\Phi.$$ 
Hence, $\om_1=\om$ is the unique lift of $\Phi$ to a $d_{X/S}$-closed $1$-form.

Next, by the definition of the connecting homomorphism, we have to lift $(f,\om_1,\om_2)$ to Cech cochain
$$(f,\wt{\om}_1,\wt{\om}_2)\in \OO(U_{12})\oplus \Om^1_X(U_1)\oplus \Om^1_X(U_2)=\CC^1(\Om^\bullet_X),$$
and consider its differential 
$(-df+\wt{\om_2}-\wt{\om}_1,d\wt{\om}_1,d\wt{\om}_2)\in \CC^2(\Om^\bullet_X)$.
It will then belong to the subcomplex $\Om^{\bullet-1}_{X/S}\ot \pi^*\Om^1_S$. Applying the projection \eqref{Om-om-map} and using Proposition \ref{H1-prop},
we arrive at the claimed description of $\nabla^{GM}([\Phi])$.
\end{proof}

\begin{remark}
The formula of Proposition \ref{GM-recipe-prop} should be compared with the following classical recipe
for computing the Gauss-Manin connection on 
$$\HH^1_{C/S_0}\simeq \coker(d:\OO_C(C-p)\to \Om^1_{C/S_0}(C-p))$$
for a usual family of curves $C\to S_0$ over a smooth even (affine or Stein) base $S_0$, 
with the relative point $p:S_0\to C$. One simply lifts a representative $\om\in\Om^1_{C/S_0}(C-p)$
to $\wt{\om}\in \Om^1_C$. Then $d\wt{\om}\mod \Om^2_{S_0}$ is in the image of $\Om^1_{C/S_0}(C-p)\ot \Om^1_{S_0}$
and is equal to $\nabla^{GM}([\om])$. 
\end{remark}

\begin{remark}
Since the functor $i^*$ induces an equivalence between $D$-modules on $S$ and on $S_0$ (see \cite[Sec.\ 2.1]{Penkov}),
the isomorphism \eqref{bos-H-iso}
 implies that a basis of $\HH^1_{X/S}$, horizontal with respect to $\nabla^{GM}$, is uniquely determined by
the corresponding basis of $\HH^1_{C/S_0}$ obtained by restriction.
\end{remark}

\subsection{Gauss-Manin-connection for supercurves of genus $1$}

Here we will describe the Gauss-Manin connection on the cohomology of the family $X\to S$ of supercurves of genus $1$ considered in Section \ref{g1-sec}.

We have a global trivialization of $\om_{X/S}$ given by $s:=\de(\th)$.
Thus, the map $f\mapsto s\cdot f$ gives an isomorphism
of $\coker(D:\OO(X-p))$ with $\HH^1_{X/S}\simeq \coker(\OO(X-p)\to \om_{X/S}(X-p))$.
The formula for the action of $D$ on $\OO(X-p)$ (see Proposition \ref{fun-prop}) immediately implies that
$$s\Psi_1,\ s\Psi_2$$
give a basis of $\HH^1$ (as an $\OO_S$-module).

\begin{theorem}\label{GM-thm}
One has 
$$\nabla^{GM}_{\varphi}(s\Psi_1)=\nabla^{GM}_{\varphi}(s\Psi_2)=0,$$
$$\nabla^{GM}_{\tau}(s\Psi_1)=s\Psi_2, \ \ \nabla^{GM}_{\tau}(s\Psi_2)=0.$$
\end{theorem}

\begin{proof}
First, we need to find for $s\Psi_1$ and $s\Psi_2$ the relative $1$-forms $\om_1,\om_2\in \Om^1_{X/S}(U_1)$ as in Proposition \ref{GM-recipe-prop}.

\begin{lemma}
We have a well-defined element 
$$\om_1=d\th\cdot (\th-\varphi\zeta_1)+(dz-d\th\cdot\th)(1-\th\varphi\zeta'_1)\in \Om^1_{X/S}(U_1),$$
such that $\de(\om_1)=s\Psi_1$ and $d_{X/S}(\om_1)=0$. 
Similarly, the element 
$$\om_2=d\th\cdot (\varphi\dot{\zeta}_1+\th\zeta'_1)+(dz-d\th\cdot \th)(\zeta'_1+\th\varphi\dot{\zeta}'_1)\in \Om^1_{X/S}(U_1),$$
satisfies $\de(\om'_1)=s\Psi_2$ and $d_{X/S}(\om_2)=0$.
\end{lemma}

\begin{proof} It is straightforward to check that $\om_i$ are invariant under the transformations \eqref{main-transformation} 
and satisfy the claimed equations.
\end{proof}

Next, we need liftings $\wt{\om}_i\in \Om^1_X(U_1)$ of $\om_i$, for $i=1,2$.
We can look for $\wt{\om}_1$ in the form
$$\wt{\om}_1=d\th\cdot (\th-\varphi\zeta_1)+(dz-d\th\cdot\th)(1-\th\varphi\zeta'_1)+d\varphi\cdot (\varphi A(z)+\th B(z))+d\tau\cdot (C(z)+\th\varphi D(z)),$$
Solving for the condition of invariance under $(z,\th)\mapsto (z+\tau+\th\varphi,\th+\varphi)$, we find a solution
$$\wt{\om}_1=d\th\cdot (\th-\varphi\zeta_1)+(dz-d\th\cdot\th)(1-\th\varphi\zeta'_1)+d\varphi\cdot \th \zeta_1-d\tau\cdot (\zeta_1+2\th\varphi\dot{\zeta}_1).$$
Now we compute the total differential (up to some terms that are not important):
$$d\wt{\om}_1\equiv [d\zeta_1+d\th\cdot \varphi\dot{\zeta}_1]\cdot d\tau
\mod((dz-d\th\cdot\th)\Om^1+\Om^2_S).
$$
Applying $\de\ot\id$ we get $s\Psi_2\ot d\tau$.

Similarly, we find 
$$\wt{\om}_2=d\th\cdot (\varphi\dot{\zeta}_1+\th\zeta'_1)+(dz-d\th\cdot \th)(\zeta'_1+\th\varphi\dot{\zeta}'_1)-d\varphi\cdot \th\dot{\zeta}_1+
d\tau(\dot{\zeta}_1+\th\varphi\ddot{\zeta}_1),$$
and the total differential is
$$d\wt{\om}_2\equiv 0 \mod((dz-d\th\cdot\th)\Om^1+\Om^2_S).$$

This finishes the proof of Theorem \ref{GM-thm}.
\end{proof}

Since $\om_{X/S}=s\cdot \OO_X$, we have $\pi_*\om_{X/S}\simeq \pi_*\OO_X$.
Thus the module of global differentials is not locally free over $\OO_S$ and is generated by
$$s, s\th\varphi=s\Psi_1\varphi.$$ 

\begin{cor}\label{per-cor}
$$e:=s\Psi_1-s\Psi_2\cdot \tau, \ f:=s\Psi_2$$
is a $\nabla^{GM}$-horizontal basis of $\HH^1_{X/S}$, extending the standard basis $(e_0,f_0)$ of $\HH^1_{C/S_0}$ 
(see \eqref{e0-f0-eq}).
We have the following relations in $\HH^1_{X/S}$:
\begin{equation}\label{per-rel1}
s\equiv s\Psi_2\varphi=f\varphi,
\end{equation}
\begin{equation}\label{per-rel2}
s\th\varphi=s\Psi_1\varphi=e\varphi+f\tau\varphi.
\end{equation}
\end{cor}

Note that the relation \eqref{per-rel1} follows from \eqref{D-Psi-eq}.

\end{document}